\documentclass[12pt]{article}

\usepackage{pdfsync}
\usepackage{amsmath,amsfonts,amssymb,graphicx,amsthm,array,color,latexsym,dsfont,amsxtra,amstext}
\usepackage[all]{xy}
\usepackage[ansinew]{inputenc}
\usepackage[margin=1in]{geometry}
\usepackage{cite}
\usepackage{authblk}
\newtheorem{thm}{Theorem}[section]
\newtheorem{cor}[thm]{Corollary}
\newtheorem{lem}[thm]{Lemma}

\theoremstyle{definition}
\newtheorem{defn}[thm]{Definition}
\theoremstyle{remark}
\newtheorem{rem}[thm]{Remark}
\newtheorem{exm}[thm]{Example}
\numberwithin{equation}{section}

\begin{document}

\title{Extensions of hom-Lie color algebras}
%\author{A. Armakan\\S. Silvestrov\thanks{Corresponding author}\\M. R. Farhangdoost Department of Mathematics, College of Sciences, Shiraz University,\\ P.O. Box 71457-44776, Shiraz, Iran \\E-mail: farhang@shirazu.ac.ir}
\author[1]{A. R. Armakan}
\author[2]{S. Silvestrov\thanks{Corresponding author: E-mail: sergei.silvestrov@mdh.se}}
\author[1]{M. R. Farhangdoost}

\affil[1]{Department of Mathematics, College of Sciences, Shiraz University, P.O. Box 71457-44776, Shiraz, Iran.}
\affil[2]{Division of Applied Mathematics,
School of Education, Culture and Communication,
 M\"{a}lardalen University, Box 883, 72123 V{\"a}ster{\aa}s, Sweden.}
\date{March 21, 2017}

%\date{}
%\dedicatory{}%
%\commby{}%
% ----------------------------------------------------------------
\maketitle
\begin{abstract} In this paper we study (non-Abelian) extensions of a given hom-Lie color algebra and provide a geometrical interpretation of extensions. In particular, we characterize an extension of a hom-Lie algebra $\mathfrak{g}$ by another hom-Lie algebra $\mathfrak{h}$ and we discuss the case where $\mathfrak{h}$ has no center. We also deal with the setting of covariant exterior derivatives, Chevalley derivative, curvature and the Bianchi identity for the possible extensions in differential geometry. Moreover, we find a cohomological obstruction to the existence of extensions of hom-Lie color algebras, i. e. we show that in order to have an extendible hom-Lie color algebra, there should exist a trivial member of the third cohomology.\\\\
\textbf{M.S.C. 2010:} 17B56, 17B75, 17B40.
\end{abstract}

\section{Introduction}
The investigations of various quantum deformations (or $q$-deformations) of Lie algebras began a period of rapid expansion in 1980's stimulated by introduction of quantum groups motivated by applications to the quantum Yang-Baxter equation, quantum inverse scattering methods and constructions of the quantum deformations of universal enveloping algebras of semi-simple Lie algebras.
Since then several other versions of $q$-deformed Lie algebras have appeared, especially in physical contexts such as string theory, vertex models in conformal field theory, quantum mechanics and quantum field theory in the context of deformations of infinite-dimensional algebras, primarily the Heisenberg algebras, oscillator algebras and Witt and Virasoro algebras
\cite{AizawaSaito,ChaiElinPop,ChaiIsLukPopPresn,ChaiKuLuk,ChaiPopPres,CurtrZachos1,DamKu,DaskaloyannisGendefVir,Hu,Kassel92,LiuKQuantumCentExt,LiuKQCharQuantWittAlg,LiuKQPhDthesis}. In these pioneering works it has been in particular descovered that in these $q$-deformations of Witt and Visaroro algebras and some related algebras, some interesting $q$-deformations of Jacobi identities, extanding Jacobi identity for Lie algebras, are satisfied. This has been one of the initial motivations for the development of general quasi-deformations and discretizations of Lie algebras of vector fields using more general $\sigma$-derivations (twisted derivations) in \cite{HLS}, and introduction of abstract quasi-Lie algebras and subclasses of quasi-Hom-Lie algebras and Hom-Lie algebras as well as their general colored (graded) counterparts in \cite{HLS,LS1,LSGradedquasiLiealg,Czech:witt,LS2}. These generalized Lie algebra structures with (graded) twisted  skew-symmetry and twisted Jacobi conditions by linear maps are tailored to encompass within the same algebraic framework such quasi-deformations and discretizations of Lie algebras of vector fields using $\sigma$-derivations, describing general descritizations and deformations of derivations with twisted Leibniz rule, and the well-known generalizations of Lie algebras such as color Lie algebras which are the natural generalizations of Lie algebras and Lie superalgebras. Quasi-Lie algebras are nonassociative algebras for which the skew-symmetry and the Jacobi identity are twisted by several deforming twisting maps and also the Jacobi identity in quasi-Lie and quasi-Hom-Lie algebras in general contains six twisted triple bracket terms.
Hom-Lie algebras is a special class of quasi-Lie algebras with the bilinear product satisfying the non-twisted skew-symmetry property as in Lie algebras, whereas the Jacobi identity contains three terms twisted by a single linear map, reducing to the Jacobi identity for ordinary Lie algebras when the linear twisting map is the identity map. Subsequently, hom-Lie admissible algebras have been considered in \cite{MS} where also the hom-associative algebras have been introduced and shown to be hom-Lie admissible natural generalizations of associative algebras corresponding to hom-Lie algebras. In \cite{MS}, moreover several other interesting classes of hom-Lie admissible algebras generalising some non-associative algebras, as well as examples of finite-dimentional hom-Lie algebras have been described.

Since these pioneering works \cite{HLS,LS1,LSGradedquasiLiealg,LS2,LS3,MS}, hom-algebra structures have become a popular area with increasing number of publications in various directions. Hom-Lie algebras, hom-Lie superalgebras and Hom-Lie color algebras are important special classes of color (graded) quasi-Lie algebras introduced first by Larsson and Silvestrov \cite{LSGradedquasiLiealg,LS2}.  Hom-Lie algebras and hom-Lie superalgebras have been studied in different aspects by Makhlouf, Silvestrov, Sheng, Ammar, Yau and other authors \cite{SB,SC,SR,AEM,MS,HomDeform,HomHopf,HomAlgHomCoalg,YN,BM,MAK,DER,
Yau:EnvLieAlg,Yau:HomolHomLie,Yau:HomBial,LarssonSigSilvJGLTA2008, RichardSilvestrovJA2008, RichardSilvestrovGLTbdSpringer2009, SigSilvGLTbdSpringer2009, SilvestrovParadigmQLieQhomLie2007},
and Hom-Lie color algebras have been considered for example in \cite{YN,CCH,spl,COH}.
We wish to mention specially \cite{AmmarMakhloufHomLieSupAlg2010}, where the constructions of Hom-Lie and quasi-hom Lie algebras based on twisted discretizations of vector fields \cite{HLS} and Hom-Lie admissible algebras have been extended to Hom-Lie superalgebras, a subclass of graded quasi-Lie algebras \cite{LS2,LSGradedquasiLiealg}.
We also wish to mention that $\mathbb{Z}_3$-graded generalizations of supersymmetry, $\mathbb{Z}_3$-graded algebras, ternary structures and related algebraic models
for classifications of elementary particles and unification problems for interactions, quantum gravity and non-commutative gauge theories \cite{Kerner7,Kerner,Kerner2,Kerner4,Kerner6} also provide interesting examples related to Hom-associative algebras, graded Hom-Lie algebras, twisted differential calculi and $n$-ary Hom-algebra structures. It would be a project of grate interest to extend and apply all the constructions and results in the present paper in the relevant contexts of the articles \cite{Kerner7,AmmarMabroukMakhloufCohomnaryHNLalg2011,AmmarMakhloufHomLieSupAlg2010, ArnlindMakhloufSilvTernHomNambuJMP2010,ArnlindMakhloufSilvnaryHomLieNambuJMP2011,ArnlindKituoniMakhloufSilv3aryCohom,
Kerner,Kerner2,Kerner4,LS2,LSGradedquasiLiealg,MS}.

In the Section \ref{sec:DefinNotsHomColorLie} of this paper hom-Lie algebras, hom-Lie superalgebras, hom-Lie color algebras and some useful related definitions are presented. In the Section \ref{sec:ExtentionsHomLie}, we introduce hom-Lie color algebra extensions with an emphasize on geometric aspects. Finally in Section \ref{sec:CohomObstruction}, introducing the Chevalley cohomology for hom-Lie color algebras, we find a cohomological obstruction to the existence of extensions.

\section{Definitions and notations for hom-Lie algebras, hom-Lie superalgebras, hom-Lie color algebras} \label{sec:DefinNotsHomColorLie}

\begin{defn} (\cite{HLS,LS1,LSGradedquasiLiealg,LS2,MS})
A hom-Lie algebra is a triple $(\mathfrak{g},[,], \alpha)$, where $\mathfrak{g} $ is a linear space equipped with a skew-symmetric bilinear map $[,]: \mathfrak{g} \times \mathfrak{g} \rightarrow \mathfrak{g}$ and a linear map $\alpha :\mathfrak{g} \rightarrow \mathfrak{g} $ such that
$$[\alpha(x),[y,z]]+[\alpha(y),[z,x]]+[\alpha(z),[x,y]]=0, $$
for all $x,y,z \in \mathfrak{g}$ , which is called hom-Jacobi identity.
\end{defn}
A hom-Lie algebra is called a multiplicative hom-Lie algebra if $\alpha$ is an algebraic morphism,
$$\alpha([x,y])=[\alpha(x),\alpha(y)],$$
for any $x,y\in \mathfrak{g}$.

We call a hom-Lie algebra regular if $\alpha$ is an automorphism.

A linear subspace $\mathfrak{h}\subseteq \mathfrak{g}$ is a hom-Lie sub-algebra of $(\mathfrak{g},[,], \alpha)$ if $$\alpha(\mathfrak{h})\subseteq \mathfrak{h},$$ and $\mathfrak{h}$ is closed under the bracket operation, i.e.
$$[x_{1},x_{2}]_{g}\in \mathfrak{h},$$
for all $x_{1},x_{2}\in \mathfrak{h}.$

Let $(\mathfrak{g},[,], \alpha)$ be a multiplicative hom-Lie algebra. Denote by $\alpha^{k}$  the $k$-times composition of $\alpha$ by itself, for any nonnegative integer $k$, i.e.
$$\alpha^{k}=\alpha \circ ... \circ \alpha ~~~~~(k-times),$$
where we define $\alpha^{0}=Id$ and $\alpha^{1}=\alpha$. For a regular hom-Lie algebra $\mathfrak{g}$, let
$$\alpha^{-k}=\alpha^{-1} \circ ... \circ \alpha^{-1}~~~~~(k-times).$$

\begin{defn}(\cite{BahturinMikhPetrZaicevIDLSbk92,MikhZolotykhCALSbk95,ScheunertGLA,ScheunertGTC})
Given a commutative group $\Gamma$ which will be in what follows referred to as the grading group, a commutation factor on $\Gamma$ with values in the multiplicative group $K\setminus \{0\}$ of a field $K$ of characteristic 0 is a map
$$\varepsilon: \Gamma \times \Gamma \rightarrow K\setminus \{0\},$$
satisfying three properties:
\begin{itemize}
  \item [1.] $\varepsilon(\alpha+\beta,\gamma)=\varepsilon(\alpha,\gamma)\varepsilon(\beta,\gamma),$
  \item [2.] $\varepsilon(\alpha,\gamma+\beta)=\varepsilon(\alpha,\gamma)\varepsilon(\alpha,\beta),$
  \item [3.] $\varepsilon(\alpha,\beta)\varepsilon(\beta,\alpha)=1.$
\end{itemize}
A $\Gamma$-graded $\varepsilon$-Lie algebra (or a color Lie algebra) is a $\Gamma$-graded linear space
$$X=\bigoplus_{\gamma\in \Gamma} X_{\gamma},$$
with a bilinear multiplication (bracket) $[.,.]:X\times X \rightarrow X$ satisfying the following properties:
\begin{itemize}
\item [1.] \textbf{Grading axiom:} $[X_{\alpha}, X_{\beta}]\subseteq X_{\alpha+\beta},$
\item [2.] \textbf{Graded skew-symmetry:} $[a,b]=-\varepsilon(\alpha,\beta)[b,a],$
\item [3.]  \textbf{Generalized Jacobi identity:}
\begin{equation}\label{eq:ColorLieJacobiId}
\varepsilon(\gamma,\alpha)[a,[b,c]]+ \varepsilon(\beta,\gamma)[c,[a,b]] +\varepsilon(\alpha,\beta)[b,[c,a]]=0
\end{equation}
\end{itemize}
for all $a\in X_{\alpha}, b\in X_{\beta}, c\in X_{\gamma}$ and $\alpha,\beta,\gamma \in \Gamma$.
\end{defn}

Let $\mathfrak{g}=\bigoplus_{\gamma\in \Gamma}\mathfrak{g}_{\gamma}$ and $\mathfrak{h}=\bigoplus_{\gamma\in \Gamma}\mathfrak{h}_{\gamma}$ be
a $\Gamma$-graded linear space. The elements of $X_{\gamma}$ are called homogenous of degree $\gamma$, for all $\gamma\in \Gamma$.
Let $\mathfrak{g}=\bigoplus_{\gamma\in \Gamma}\mathfrak{g}_{\gamma}$ and $\mathfrak{h}=\bigoplus_{\gamma\in \Gamma}\mathfrak{h}_{\gamma}$ be two $\Gamma$-graded linear spaces. A linear mapping $f:\mathfrak{g}\rightarrow \mathfrak{h}$ is said to be graded or homogenous of the degree $\mu\in \Gamma$ if $$f(\mathfrak{g}_{\gamma})\subseteq \mathfrak{h}_{\gamma+\mu},$$
for all $\gamma\in \Gamma$. A graded linear mapping $f$ is said to be homogenous of degree zero if $$f(\mathfrak{g}_{\gamma})\subseteq \mathfrak{h}_{\gamma},$$ holds for any $\gamma\in \Gamma$.
Sometimes such $f$ are said to be even.

A special class of color quasi-Lie algebras \cite{LSGradedquasiLiealg,LS2}
are coloured hom-Lie algebras.

\begin{defn} (\cite{LSGradedquasiLiealg,LS2,YN})
A hom-Lie color algebra is a quadruple $(\mathfrak{g},[.,.],\varepsilon,\alpha)$ consisting of a $\Gamma$-graded linear space $\mathfrak{g}=\bigoplus_{\gamma\in \Gamma}\mathfrak{g}_{\gamma}$, a bi-character $\varepsilon$, a graded bilinear mapping $[.,.]:\mathfrak{g}\times \mathfrak{g}\rightarrow \mathfrak{g}$ (i.e. $[\mathfrak{g}_{a},\mathfrak{g}_{b}]\subseteq \mathfrak{g}_{a+b}$, for all $a,b \in \Gamma$) and graded homomorphism $\alpha:\mathfrak{g}\rightarrow \mathfrak{g}$ of grading degree zero ($\alpha (\mathfrak{g}_{\gamma})\subseteq \mathfrak{g}_{\gamma}$ for all $\gamma \in \Gamma$) such that for homogeneous elements $x,y,z\in \mathfrak{g}$ we have
\begin{itemize}
  \item [1.] \textbf{$\varepsilon$-skew symmetric:} $[x,y]=-\varepsilon(x,y)[y,x],$
  \item [2.] \textbf{$\varepsilon$-hom-Jacobi identity:}
  \begin{equation*} \label{eq:ColorHomLieJacobiId}
  \sum_{cyclic\{x,y,z\}}\varepsilon(z,x)[\alpha(x),[y,z]]=0.
  \end{equation*}
\end{itemize}
\end{defn}
In particular, if $\alpha$ is a morphism of color Lie algebras i.e.
$$\alpha\circ [.,.]=[.,.]\circ \alpha^{\otimes2},$$
then we call $(\mathfrak{g},[,],\varepsilon,\alpha)$, a multiplicative hom-Lie color algebra.\\

\begin{exm} {\rm (\cite{COH})}
Let $(\mathfrak{g},[,],\varepsilon)$ be a color Lie algebra and $\alpha$ be a Lie color algebra morphism. Then
$$(\mathfrak{g},[,]_{\alpha}:=\alpha\circ [.,.],\varepsilon,\alpha)$$
 is a multiplicative hom-Lie color algebra.
\end{exm}

\begin{defn} (\cite{COH})
Let $(\mathfrak{g},[,],\varepsilon,\alpha)$ be a hom-Lie color algebra. For any nonnegative integer $k$, a linear map
$$D: \mathfrak{g} \rightarrow \mathfrak{g}$$
of degree $d$ is called a homogenous $\alpha^{k}$-derivation of the multiplicative hom-Lie color algebra $(\mathfrak{g},[,],\varepsilon,\alpha)$, if
\begin{itemize}
\item[1.] $D(\mathfrak{g}_{\gamma})\subseteq \mathfrak{g}_{\gamma+d}$, for all $\gamma\in \Gamma$
\item[2.] $[D,\alpha]=0$, i.e. $D\circ \alpha= \alpha \circ D,$
\item[3.] $D([x,y]_{g})=[D(x),\alpha^{k}(y)]_{\mathfrak{g}}+\varepsilon(d,x)[\alpha^{k}(x),D(y)]_{\mathfrak{g}}$, for all $x,y\in \mathfrak{g}$.
\end{itemize}
\end{defn}
Denote by $Der^{\gamma}_{\alpha^{k}}(\mathfrak{g})$ the set of all homogenous $\alpha^{k}$-derivations of the multiplicative hom-Lie color algebra $(\mathfrak{g},[,], \alpha)$.
The space
$$Der(\mathfrak{g})=\bigoplus_{k\geq 0}Der_{\alpha^{k}}(\mathfrak{g}),$$
provided with the color-commutator and the following linear map
$$\widetilde{\alpha}:Der(\mathfrak{g})\rightarrow Der(\mathfrak{g}), \hspace{1cm} \widetilde{\alpha}(D)=D\circ \alpha,$$
is a color Lie algebra.

For any $x\in \mathfrak{g}$ satisfying $\alpha(x)=x$, define $$ad_{k}(x):\mathfrak{g} \rightarrow \mathfrak{g}$$ by
$$ad_{k}(x)(y)=[\alpha^{k}(y),x]_{\mathfrak{g}},$$
for all $y\in \mathfrak{g}$.

It is shown in \cite{COH} that $ad_{k}(x)$ is a $\alpha^{k+1}$-derivation, which we call an inner $\alpha^{k}$-derivation.
So
$$Inn_{\alpha^{k}}(\mathfrak{g})=\{[\alpha^{k-1}(.),x]_{\mathfrak{g}}|x\in \mathfrak{g}, \alpha(x)=x\}.$$

\section{Extensions of hom-Lie color algebras} \label{sec:ExtentionsHomLie}
In this section we clarify what we mean by an extension of a hom-Lie color algebra. Although one can see that extensions of a given Abelian hom-Lie color algebra is characterized by elements of its second cohomology group, we concentrate on some geometric aspects in this research. The cohomology has been studied for Lie superalgebras and color Lie algebras \cite{PiontkovskiSilvestrovC3dCLA,ChenPetitOystaeyenCOHCHLA,ScheunertCOH2,ScheunertZHA} and hom-Lie color algebras \cite{COH}.
\begin{defn}
Let $\mathfrak{g} $,$\mathfrak{h} $ be two hom-Lie color algebras, we call $\mathfrak{e} $ an extension of the hom-Lie color algebra $\mathfrak{g}$ by $\mathfrak{h} $, if there exists a short exact sequence
$$0\rightarrow \mathfrak{h}\rightarrow \mathfrak{e} \rightarrow \mathfrak{g}\rightarrow 0,$$
of hom-Lie color algebras and their morphisms.
\end{defn}

Two extensions
$$\xymatrix{0  \ar[r] & \mathfrak{h} \ar[r]^{i_{k}} & \mathfrak{e}_{k} \ar[r]^{p_{k}} & \mathfrak{g} \ar[r]& 0}(k=1,2)$$
are equivalent if there is an isomorphism $f:\mathfrak{e}_{1} \rightarrow \mathfrak{e}_{2}$ such that $f\circ i_{1}=i_{2}$ and $p_{2}\circ f=p_{1}$.

We want to study the possible extensions, so suppose there exists an extension
$$\xymatrix{0  \ar[r] & \mathfrak{h} \ar[r]^{i} & \mathfrak{e} \ar[r]^{p} & \mathfrak{g} \ar[r]& 0}$$
and let $s:\mathfrak{g}\rightarrow \mathfrak{e}$ be a graded linear map of even degree such that $p\circ s=Id_{\mathfrak{g}}$.
We define
\begin{align}\label{fi}
\varphi &: \mathfrak{g}\rightarrow Der_{\alpha^{k}}(\mathfrak{h}),\\
\varphi_{x}(y)&=[\alpha^{k-1}(s(x)),y],\nonumber
\end{align}
and
\begin{align}\label{ro}
\rho:& \bigwedge^{2}_{\Gamma-graded}\mathfrak{g}\rightarrow \mathfrak{h},\\
\rho(x,y)=&[s(x),s(y)]-s([x,y]).\nonumber
\end{align}
The following lemma shows some properties of the above maps which we will use later on this research.
\begin{lem}\label{lem1}
The maps $\varphi$ and $\rho$ defined in \eqref{fi} and \eqref{ro} satisfy
\begin{equation}\label{1}
[\varphi_{x},\varphi_{y}]-\varphi_{[x,y]}=ad_{k-1}(\rho(x,y)),
\end{equation}
\begin{equation}\label{2}
\sum_{cyclic\{x,y,z\}}\varepsilon(x,z)(\varphi_{x}\rho(y,z)-\rho([x,y],z))=0.
\end{equation}

\end{lem}
\begin{proof}
First note that $\varphi_{x}=ad_{k-1}(s(x))$. So we have
\begin{align*}
&[\varphi_{x},\varphi_{y}]-\varphi_{[x,y]}=[ad_{k-1}(s(x)),ad_{k-1}(s(y))]\\&-ad_{k-1}(s([x,y]))=ad_{k-1}([s(x),s(y)]-s([x,y]))\\
&=ad_{k-1}(\rho(x,y)),
\end{align*}
which proves the first equality. For the second equality we have
\begin{align*}
&\sum_{cyclic\{x,y,z\}}\varepsilon(x,z)[\varphi_{x}\rho(y,z)-\rho([x,y],z),w]\\
&=\sum_{cyclic\{x,y,z\}}\varepsilon(x,z)(\varphi_{x}[\rho(y,z),w]\\&-\varepsilon(x,y+z)[\rho(y,z),\varphi_{x}(w)]-[\rho([x,y],z),w])\\
&=\sum_{cyclic\{x,y,z\}}\varepsilon(x,z) (\varphi_{x}[\varphi_{y},\varphi_{z}]-\varphi_{x}\varphi_{[y,z]}\\
&-\varepsilon(x,y+z)[\varphi_{y},\varphi_{z}]\varphi_{x} +\varepsilon(x,y+z)\varphi_{[y,z]}\varphi_{x} \\ &-[\varphi_{[x,y]},\varphi_{z}]+\varphi_{[[x,y],z]})w\\
&=\sum_{cyclic\{x,y,z\}}\varepsilon(x,z)([\varphi_{x},[\varphi_{y},\varphi_{z}]] -[\varphi_{x},\varphi_{[y,z]}]\\
&-[\varphi_{[x,y]},\varphi_{z}]+\varphi_{[[x,y],z]})w =0.
\end{align*}
\end{proof}
Therefore, using $\varphi$ and $\rho$ which satisfy \eqref{1} and \eqref{2}, the hom-Lie color algebra structure on $\mathfrak{e}=\mathfrak{h}\oplus s(\mathfrak{g})$ will be in the following form:
\begin{align*}
&[y_{1}+s(x_{1}),y_{2}+s(x_{2})]_{\mathfrak{e}} =([y_{1},y_{2}]_{\mathfrak{h}}+\varphi_{x_{1}}y_{2} \\ &-\varepsilon(y_{1},x_{2})\varphi_{x_{2}}y_{1}+\rho(x_{1},x_{2}))+s([x_{1},x_{2}]_{\mathfrak{g}}).
\end{align*}

We can see $\varphi$ as a connection in the sense of \cite{KMP} and $\rho$ as its curvature. Moreover, in order to give \eqref{2} a more geometric image, we have to use a graded version of the Chevalley coboundary operator which makes it a $\Gamma$-graded exterior covariant derivative. There is a notation that we will need throughout this research, which is a modification of the one introduced in \cite{ScheunertGTC}:
$$\varepsilon_{k}:\mathcal{S}_{k}\times \Gamma^{k}\rightarrow K,$$
$$\varepsilon_{k}(\sigma^{-1},a_{1},...,a_{k})=\prod_{\substack{i<j,\\ \sigma(i)>\sigma(j)}} \varepsilon(a_{i},a_{j}),$$
for all $a_{1},...,a_{k}\in \Gamma$ and $\sigma \in \mathcal{S}_{k}$. If we write
$$\sigma (a)=(a_{\sigma(1)},...,a_{\sigma (k)}),$$
for $a=(a_{1},...,a_{k})\in (\Gamma)^{k}$, then we have
$$\varepsilon_{k}(\sigma \circ \tau,a)=\varepsilon_{k}(\sigma,a)\varepsilon_{k}(\tau,\sigma (a)).$$
Moreover, we simply show the degree of any homogenous element $x_{i}\in \mathfrak{g}$ by $\overline{x_{i}}$ through the rest.

Now, for a linear space $V$, the space of graded $p$-linear skew symmetric maps $\mathfrak{g}^{p}\rightarrow V$ of degree $w$ is denoted by $A^{p,w}_{gskew}(\mathfrak{g},V)$, i.e. the space of maps of the following form
$$\psi(x_{1},...,x_{p}) \in V_{w+\overline{x_{1}}+...+\overline{x_{p}}},$$
$$\psi(x_{1},...,x_{p})=-\varepsilon(x_{i},x_{i+1})\psi(x_{1},...,x_{i+1},x_{i},...,x_{p}),$$
for all homogenous elements $x_{i}$, $i=1,...,k$.
We see that when $\psi\in A^{p,w}_{gskew}(\mathfrak{g},\mathfrak{h})$, we have
$$\varepsilon_{p}(\sigma,x)\psi(x_{\sigma (1)},...,x_{\sigma (p)})=sign(\sigma)\psi(x_{1},...,x_{p}).$$

\begin{defn}
Using a connection like the one defined in \eqref{fi}, the $\Gamma$-graded version of the covariant exterior derivative is defined by
$$\delta_{\varphi}:A^{p,w}_{gskew}(\mathfrak{g},\mathfrak{h})\rightarrow A^{p,w}_{gskew}(\mathfrak{g},\mathfrak{h}),$$
\begin{align*}
(\delta_{\varphi}\psi)(x_{0},...,x_{p})&=\sum_{i=0}^{p}(-1)^{i}\theta_{i}(x)\varphi_{x_{i}}(\psi(x_{0},...,\widehat{x_{i}},...,x_{p}))\\
&+\sum_{i<j}(-1)^{j}\theta_{ij}(x)\psi(\alpha(x_{0}),...,\alpha(x_{i-1}),\\
&[x_{i},x_{j}],\alpha(x_{i-1}),...,\widehat{x_{j}},...,\alpha(x_{p})),
\end{align*}
for all homogenous elements $x_{i}$, $i=1,...,k$ where $$\theta_{i}(x)=\varepsilon(\overline{x_{1}}+...+\overline{x_{i-1}},\overline{x_{i}})$$ and $$\theta_{ij}(x)=\varepsilon(\overline{x_{i+1}}+...+\overline{x_{j-1}},\overline{x_{j}}).$$
\end{defn}
We see that for $\psi \in A^{p,w}_{gskew}(\mathfrak{g},\mathfrak{h})$ and $\zeta \in A^{q,z}_{gskew}(\mathfrak{g},\mathds{R})$ which is a form of degree $q$ and weight $z$, the operator $\delta_{\varphi}$ satisfies the condition for a graded covariant exterior derivative, i.e.
$$\delta_{\varphi}(\zeta\wedge\psi)=\delta\zeta\wedge\psi+\varepsilon(\delta_{\varphi},q)\zeta\wedge\delta_{\varphi}\psi,$$
where
\begin{align*}
(\delta\zeta)(x_{0},...,x_{q})=&\sum_{i<j}(-1)^{j}\theta_{ij}(x)\psi(\alpha(x_{0}),...,\alpha(x_{i-1}),\\ &[x_{i},x_{j}],\alpha(x_{i-1}),...,\widehat{x_{j}},...,\alpha(x_{p})),
\end{align*}
is the color version of the Chevalley coboundary operator and the module structure is given by
\begin{align*}
 &(\zeta\wedge\psi)(x_{1},...,x_{p+q})  =\frac{1}{p!q!}\sum_{\sigma\in \mathcal{S}_{p+q}}sign(\sigma)\varepsilon(y,\eta_{q}(\sigma,x)) \\
 &\varepsilon_{p+q}(\sigma,x) \zeta(x_{\sigma 1},..., x_{\sigma q})\psi(x_{\sigma(q+1)},...,x_{\sigma (q+p)}),
\end{align*}
where $\eta_{i}(\sigma,x)=\overline{x_{\sigma 1}}+...+\overline{x_{\sigma i}}$.

Also, for $\psi\in A^{p,w}_{gskew}(\mathfrak{g},\mathfrak{h})$ and $\zeta\in A^{q,z}_{gskew}(\mathfrak{g},\mathfrak{h})$, we have
\begin{align*}
 & [\psi,\zeta]_{\wedge}(x_{1},...,x_{p+q})  = \frac{1}{p!q!}\sum_{\sigma}sign(\sigma)\varepsilon(z,\eta_{p}(\sigma,x)) \\
 & \varepsilon_{p+q}(\sigma,x) [\psi(x_{\sigma(1)},...,x_{\sigma (p)}), \zeta(x_{\sigma(p+1)},...,x_{\sigma(p+q)})]_{\mathfrak{h}}.
\end{align*}

Thus, we can see $s$ as a connection in the sense of the horizontal lift of vector fields on base of a bundle. Moreover, $\varphi$ is an induced connection. See \cite{KMP} for more background information.

So the formula \eqref{2} will become the Bianchi identity
$$\delta_{\varphi} \rho=0.$$
Moreover, we deduce that
\begin{equation}\label{209}
\delta_{\varphi}\delta_{\varphi}(\psi)=[\rho,\psi]_{\wedge}, \psi\in A^{p,w}_{gskew}(\mathfrak{g},\mathfrak{h}).
\end{equation}

Therefore, if $\varphi:\mathfrak{g}\rightarrow Der_{\alpha^{k}}(\mathfrak{h})$ is a homomorphism of color hom- Lie algebras or $$\varphi:\mathfrak{g}\rightarrow End(\mathfrak{h}),$$ is a representation in a graded vector space, \eqref{209} shows the hom-color version application of the Chevalley cohomology.

We can now complete the hom-Lie color algebra structure on $\mathfrak{e}=\mathfrak{h}\oplus s(\mathfrak{g})$ as follows
\begin{align}\label{STR}
&[y_{1}+s(x_{1}),y_{2}+s(x_{2})]=([y_{1},y_{2}]+\varphi_{x_{1}}y_{2}\\ \nonumber
&-\varepsilon(y_{1},x_{2})\varphi_{x_{2}}y_{1}+\rho(x_{1},x_{2}))+s([x_{1},x_{2}]).
\end{align}
One can check that \eqref{STR} gives $\mathfrak{h}\oplus s(\mathfrak{g})$ a hom-Lie color algebra structure, if $\varphi$ and $\rho$ satisfy \eqref{1} and \eqref{2}.
If we put $s'=s+b$ instead of $s$ , where $b:\mathfrak{g} \rightarrow \mathfrak{h} $ is an even linear map, we will have
\begin{align*}
\varphi'_{x}(y)&=[\alpha(s(x)+b(x)),y]=[\alpha(s(x)),y]+[\alpha(b(x)),y]\\&= \varphi_{x}+ad_{k-1}^{\mathfrak{h}}(b(x))
\end{align*}
and
\begin{align*}
\rho'(x,y)&=\rho(x,y)+\varphi_{x}b(y)-\varphi_{y}b(x)-b([x,y])\\&+[b(x),b(y)]\\
&=\rho(x,y)+\delta_{\varphi}b(x,y)+[b(x),b(y)],
\end{align*}
i.e.
$$\rho'=\rho+\delta_{\varphi}b+\frac{1}{2}[b,b]_{\wedge}.$$
Thus, so far we have proved the following theorem.
\begin{thm}\label{th}
Let $\mathfrak{g},\mathfrak{h}$ be two hom-Lie color algebras. The isomorphism classes of extensions of $\mathfrak{g}$ on $\mathfrak{h}$ i.e. the short exact sequences of the form
$$0\rightarrow \mathfrak{h}\rightarrow \mathfrak{e} \rightarrow \mathfrak{g}\rightarrow 0,$$
is in one to one correspondence with the data of the following form:\\
An even linear map $\varphi:\mathfrak{g}\rightarrow Der_{\alpha^{k}}(\mathfrak{h})$ and a graded even skew symmetric bilinear map $\rho:\mathfrak{g}\times \mathfrak{g}\rightarrow\mathfrak{h}$ such that
\begin{equation}\label{14}
[\varphi_{x},\varphi_{y}]-\varphi_{[x,y]}=ad_{k-1}(\rho(x,y)),
\end{equation}
and
\begin{equation}\label{15}
\sum_{cyclic\{x,y,z\}}\varepsilon(x,z)(\varphi_{x}\rho(y,z)-\rho([x,y],z))=0,
\end{equation}
or in other words $\delta_{\varphi}\rho=0.$
The extension which corresponds to $\varphi$ and $\rho$ is the vector space $\mathfrak{e}=\mathfrak{h}\oplus \mathfrak{g}$ whose hom-Lie color algebra structure is given by
\begin{align*}
&[y_{1}+s(x_{1}),y_{2}+s(x_{2})]_{\mathfrak{e}}=([y_{1},y_{2}]_{\mathfrak{h}}+\varphi_{x_{1}}y_{2}\\
&-\varepsilon(x_{2},y_{1})\varphi_{x_{2}}y_{1}+\rho(x_{1},x_{2}))+[x_{1},x_{2}]_{\mathfrak{g}},
\end{align*}
and its short exact sequence is
$$\xymatrix{0  \ar[r] & \mathfrak{h} \ar[r]^{\!\!\!\!\!\!\!\!\!\!i_{1}} & \mathfrak{h}\oplus \mathfrak{g}=\mathfrak{e} \ar[r]^{~~~~~~~~~~~~~~~~~~~~~~~~~~~~~pr_{2}} & \mathfrak{g} \ar[r]& 0}.$$
Two data $(\varphi,\rho)$ and $(\varphi',\rho')$ are equivalent if there exists a linear map $b:\mathfrak{g} \rightarrow \mathfrak{h}$ of degree zero such that
$$\varphi'_{x}= \varphi_{x}+ad_{k-1}^{\mathfrak{h}}(b(x)),$$
and
\begin{align*}
\rho'(x,y)&=\rho(x,y)+\varphi_{x}b(y)-\varepsilon(x,y)\varphi_{y}b(x)-b([x,y])\\&+[b(x),b(y)]\\
&=\rho(x,y)+\delta_{\varphi}b(x,y)+[b(x),b(y)].
\end{align*}
So the corresponding equivalence will be
$$\mathfrak{e}=\mathfrak{h}\oplus \mathfrak{g}\rightarrow \mathfrak{h}\oplus \mathfrak{g}=\mathfrak{e}'$$
$$y+x\mapsto y-b(x)+x.$$
Furthermore, the datum $(\varphi,\rho)$ represents a split extension if and only if $(\varphi,\rho)$ corresponds to a datum of the form $(\varphi',0)$ (so that $\varphi'$ is an isomorphism). In this case there exists a linear map $b:\mathfrak{g}\rightarrow \mathfrak{h}$ such that
$$\rho=-\delta_{\varphi}b-\frac{1}{2}[b,b]_{\wedge}.$$
\end{thm}
It may be a good idea to illustrate the motivation of the above theorem by an example below.
\begin{exm}
Let $\pi: B\rightarrow M=\frac{B}{G}$ be a principal bundle with structure group $G$, i.e. $B$ is a manifold with a free right action of a Lie group $G$ and $\pi$ is the projection on the orbit space $M=\frac{B}{G}$. Denote by $\mathfrak{g}=\mathfrak{X}(M)$ the Lie algebra of the vector fields on $M$, by $\mathfrak{e}=\mathfrak{X}(B)^{G}$ the Lie algebra of $G$-invariant vector fields on $B$ and by $\mathfrak{X}_{v}(B)^{G}$ the ideal of the $G$-invariant vertical vector fields of $\mathfrak{e}$. We have a natural homomorphism $\pi_{*}: \mathfrak{e} \rightarrow \mathfrak{g}$ with the kernel $\mathfrak{h}$, i.e. $\mathfrak{e}$ is an extension of $\mathfrak{g}$ by means $\mathfrak{h}$.

Note that we have a $C^{\infty}(M)$-module structure on $\mathfrak{g},\mathfrak{e},\mathfrak{h}$. In particular, $\mathfrak{h}$ is a Lie algebra over $C^{\infty}(M)$. The extension
$$0\rightarrow \mathfrak{h}\rightarrow \mathfrak{e} \rightarrow \mathfrak{g}\rightarrow 0,$$
is also an extension of $C^{\infty}(M)$-modules. Now assume that the section $s:\mathfrak{g}\rightarrow \mathfrak{e}$ is a homomorphism of $C^{\infty}(M)$-modules. Then it can be considered as a connection in the principal bundle $\pi$, and the $\mathfrak{h}$-valued 2-form $\rho$ as its curvature. This geometric concept example is a very good guideline for our approach which is dealt with in \cite{ma1,ma2}. It works also for color hom-Lie algebras. For more background information, one can see \cite{KMP}, section 9. This kind of examples with the analogy to the differential geometry was first introduced in \cite{Lecomte} and has been also used in the theory of Lie algebroids, see \cite{MAc}.
\end{exm}
In the special case of Theorem \ref{th}, we have:
\begin{cor}\label{cor1}
Let $\mathfrak{g},\mathfrak{h}$ be two hom-Lie color algebras such that $\mathfrak{h}$ has no center, i.e. $$Z(\mathfrak{h})=0.$$
Then the extensions of $\mathfrak{g}$ by $\mathfrak{h}$ is in one to one correspondence with the isomorphisms of the form
$$\bar{\varphi}:\mathfrak{g}\rightarrow out(\mathfrak{h})=\frac{Der_{\alpha^{k}}(\mathfrak{h})}{Inn_{\alpha^{k}}(\mathfrak{h})}.$$
\end{cor}
\begin{proof}
If $(\varphi,\rho)$ is a datum, the map $\bar{\varphi} :\mathfrak{g} \rightarrow \frac{Der_{\alpha^{k}}(\mathfrak{h})}{Inn_{\alpha^{k}}(\mathfrak{h})}$ defined by
$$\xymatrix{\mathfrak{g} \ar[r]^{\!\!\!\!\!\varphi} & Der_{\alpha^{k}}(\mathfrak{h}) \ar[r]^{\pi} & \frac{Der_{\alpha^{k}}(\mathfrak{h})}{Inn_{\alpha^{k}}(\mathfrak{h})}},$$
$$\bar{\varphi}=\pi \circ \varphi,$$
is a hom-Lie color algebra homomorphism, because
\begin{align}\label{18}
\bar{\varphi}_{[x,y]}&=\pi(\varphi_{[x,y]})=\pi([\varphi_{x},\varphi_{y}]-ad_{k-1}(\rho(x,y))\\ \nonumber
&=\pi([\varphi_{x},\varphi_{y}])=[\pi\circ\varphi_{x},\pi\circ\varphi_{y}]=[\bar{\varphi}_{x},\bar{\varphi}_{y}]. \nonumber
\end{align}
Conversely, suppose we have the map $\bar{\varphi}$. A linear lift $$\varphi:\mathfrak{g}\rightarrow Der_{\alpha^{k}}(\mathfrak{h})$$ can be considered. Since $\bar{\varphi}$ is a hom-Lie color algebra homomorphism and $\mathfrak{h}$ has no center, there exists a skew symmetric unique linear map $\rho: \mathfrak{g} \times \mathfrak{g} \rightarrow \mathfrak{h}$ such that
$$[\varphi_{x},\varphi_{y}]-\varphi_{[x,y]}=ad_{k-1}(\rho(x,y)).$$
So the equality \eqref{1} is fulfilled. Also it is easy to obtain \eqref{2}.
\end{proof}

\section{Cohomological obstruction to existence of extensions} \label{sec:CohomObstruction}
In this section we present a proposition which shows that if there exists a hom-Lie color algebra extension, there should be a trivial member of the third cohomology. We have to give some remarks first.
\begin{rem}
The hom-Lie color algebra $\mathfrak{h}$ is a $Der_{\alpha^{k}}(\mathfrak{h})$-module with the multiplication rule
$$Der_{\alpha^{k}}(\mathfrak{h})\times \mathfrak{h}\rightarrow \mathfrak{h},(h,x)\mapsto h(x),$$
and $Z(\mathfrak{h})$ is a submodule of $\mathfrak{h}$ with this multiplication, i.e. $h(x)\in Z(\mathfrak{h})$, for all $x\in Z(\mathfrak{h})$ , $h\in Der_{\alpha^{k}}(\mathfrak{h})$, because
$$[h(x),y]=h([x,y])-[x,h(y)]=0,$$
for all $y\in \mathfrak{h}$. Thus $h(x)\in Z(\mathfrak{h})$.
Also for all $\bar{h}\in \frac{Der_{\alpha^{k}}(\mathfrak{h})}{Inn_{\alpha^{k}}(\mathfrak{h})}$, there exists some $h\in Der_{\alpha^{k}}(\mathfrak{h})$ such that $\bar{h}=[h]$ and one can see $Z(\mathfrak{h})$ as a module on $\frac{Der_{\alpha^{k}}(\mathfrak{h})}{Inn_{\alpha^{k}}(\mathfrak{h})}$; it is sufficient to define the multiplication in the following way
$$\bar{h}.x=h(x),$$
for all $x\in Z(\mathfrak{h})$ and $\bar{h}\in \frac{Der_{\alpha^{k}}(\mathfrak{h})}{Inn_{\alpha^{k}}(\mathfrak{h})}$. Note that this definition is well defined since for $\bar{h}=[h']$ we have
$$h'=h+ad_{k-1}(a),$$
so $h'(x)=h(x)+ad_{k-1}(a)(x)=h(x)+[a,x]=h(x),$ since $x$ is in the center of $\mathfrak{h}$ and $a\in \mathfrak{h}$.
Now, using the module structure of $Z(\mathfrak{h})$ on $\frac{Der_{\alpha^{k}}(\mathfrak{h})}{Inn_{\alpha^{k}}(\mathfrak{h})}$, we can give $\mathfrak{g}$ a module structure by the map $\bar{\varphi}$, i.e. for $c\in\mathfrak{g}$ and $x\in Z(\mathfrak{h})$ we put
$$cx=\bar{\varphi}(c)x~~~~~~~~~.$$

\end{rem}
\begin{rem}
For hom-Lie color algebra homomorphism $\bar{\varphi} :\mathfrak{g} \rightarrow \frac{Der_{\alpha^{k}}(\mathfrak{h})}{Inn_{\alpha^{k}}(\mathfrak{h})}$, if $V$ is a vector space which has a $\frac{Der_{\alpha^{k}}(\mathfrak{h})}{Inn_{\alpha^{k}}(\mathfrak{h})}$-module structure, one can consider the space of all $k$-linear forms on $\mathfrak{g}$ with values in $\mathfrak{h}$ which is denoted by $\bigwedge^{k}(\mathfrak{g},\mathfrak{h})$.\\
We can construct $\delta_{\bar{\varphi}}$ like $\delta_{\varphi}$. First the exterior multiplication in $\bar{\varphi}$ is defined. For all $\psi\in \bigwedge^{k}(\mathfrak{g},\mathfrak{h})$, $\bar{\varphi}\wedge\psi$ is in $\bigwedge^{k+1}(\mathfrak{g},\mathfrak{h})$ and acts in the following way
\begin{align*}
&(\bar{\varphi}\wedge\psi)(x_{0},...,x_{k})
=\frac{1}{2!k!}\sum_{\sigma\in \mathcal{S}_{k+2}} sign(\sigma) \varepsilon(y,\eta_{k}(\sigma,x)) \\ &\varepsilon_{k+2}(\sigma,x)\bar{\varphi}_{x_{\sigma(i)}} \psi(x_{\sigma(2)},...,x_{\sigma (k+2)}),
\end{align*}

and
$$\delta_{\bar{\varphi}}:\bigwedge^{k}(\mathfrak{g},\mathfrak{h})\rightarrow \bigwedge^{k+1}(\mathfrak{g},\mathfrak{h}),$$
\begin{align*}
(\delta_{\bar{\varphi}}\psi)(x_{0},...,x_{k})&=\sum_{i=0}^{k}(-1)^{i}\theta_{i}(x)\bar{\varphi}_{x_{i}}(\psi(x_{0},...,\widehat{x_{i}},...,x_{k}))\\
&+\sum_{i<j}(-1)^{j}\theta_{ij}(x)\psi(\alpha(x_{0}),...,\alpha(x_{i-1}),\\
&[x_{i},x_{j}],\alpha(x_{i-1}),...,\widehat{x_{j}},...,\alpha(x_{p})).
\end{align*}
In the special case for $V=Z(\mathfrak{h})$, $Z(\mathfrak{h})$ is a $\frac{Der_{\alpha^{k}}(\mathfrak{h})}{Inn_{\alpha^{k}}(\mathfrak{h})}$-module and if we consider $$\varphi:\mathfrak{g}\rightarrow Der_{\alpha^{k}}(\mathfrak{h}),$$ to be such that $\bar{\varphi}=\pi \circ \varphi$, where $$\pi:Der_{\alpha^{k}}(\mathfrak{h})\rightarrow \frac{Der_{\alpha^{k}}(\mathfrak{h})}{Inn_{\alpha^{k}}(\mathfrak{h})}.$$
 Then $Z(\mathfrak{h})$ is also a $Der_{\alpha^{k}}(\mathfrak{h})$-module and
$$\delta_{\varphi}:\bigwedge^{k}(\mathfrak{g},Z(\mathfrak{h}))\rightarrow \bigwedge^{k+1}(\mathfrak{g},Z(\mathfrak{h}))$$
is defined too and in this case we have $\delta_{\varphi}=\delta_{\bar{\varphi}}$, since $\bar{\varphi}_{x_{i}}=[\varphi_{x_{i}}]$. Note that the multiplication rule between $\frac{Der_{\alpha^{k}}(\mathfrak{h})}{Inn_{\alpha^{k}}(\mathfrak{h})}$ and $Z(\mathfrak{h})$ is
$$[h]x=h(x), h\in Der_{\alpha^{k}}(\mathfrak{h}).$$

Since $Z(\mathfrak{h})$ is the center of $\mathfrak{h}$, the operator $\delta_{\bar{\varphi}}$ or $\delta_{\varphi}$ satisfies $\delta_{\varphi}\circ \delta_{\varphi}=0$, i.e.
\begin{align*}
&\delta_{\varphi}\circ \delta_{\varphi}(\psi)(x_{1},...,x_{k+1})=[\rho,\psi]_{\wedge}(x_{1},...,x_{k+1})\\
&=\frac{1}{2!k!}\sum_{\sigma\in \mathcal{S}_{k+2}}sign(\sigma)\varepsilon(z,\eta_{p}(\sigma,x))\varepsilon_{k+2}(\sigma,x) \\ &[\psi(x_{\sigma(2)},x_{\sigma(2)}), \zeta(x_{\sigma(3)},...,x_{\sigma(k+2)})]=0,
\end{align*}
for $\psi \in \bigwedge^{k+1}(\mathfrak{g},Z(\mathfrak{h}))$ and $x_{1},...,x_{k+1}\in \mathfrak{g}$.
\end{rem}
\begin{thm}
Let $\mathfrak{g},\mathfrak{h}$ be two hom-Lie color algebras and $\bar{\varphi}:\mathfrak{g}\rightarrow \frac{Der_{\alpha^{k}}(\mathfrak{h})}{Inn_{\alpha^{k}}(\mathfrak{h})}$ a hom-Lie color algebra homomorphism. Then the followings are equivalent:
\begin{itemize}
\item[1.] For any linear lift $\varphi:\mathfrak{g}\rightarrow Der_{\alpha^{k}}(\mathfrak{h})$ of $\bar{\varphi}$, one can find a linear graded map
    $$\rho:\bigwedge^{2}_{graded}\mathfrak{g}\rightarrow\mathfrak{h},$$
    of even degree such that
    $$[\varphi_{x},\varphi_{y}]-\varphi_{[x,y]}=ad_{k-1}(\rho(x,y)).$$
    In this case the $\delta_{\bar{\varphi}}$-cohomology classes $\lambda$ will be trivial in $H^{3}(\mathfrak{g},Z(\mathfrak{h}))$, where
    $$\lambda=\lambda(\varphi,\rho):=\delta_{\varphi}(\rho):\bigwedge^{3}\mathfrak{g}\rightarrow Z(\mathfrak{h}).$$

\item[2.] There exists an extension $0\rightarrow \mathfrak{h}\rightarrow \mathfrak{e} \rightarrow \mathfrak{g}\rightarrow 0$ which induces the homomorphism $\bar{\varphi}$. In this case all the extensions $0\rightarrow \mathfrak{h}\rightarrow \mathfrak{e} \rightarrow \mathfrak{g}\rightarrow 0$ inducing $\bar{\varphi}$ will be parameterized by $H^{2}(\mathfrak{g},Z(\mathfrak{h}))$, where $H^{2}(\mathfrak{g},Z(\mathfrak{h}))$ is the second cohomology space of $\mathfrak{g}$ with values in $Z(\mathfrak{h})$ which here is considered as a $\mathfrak{g}$-module by $\bar{\varphi}$.
\end{itemize}
\end{thm}
\begin{proof}
Using the calculations in the proof of Corollary \ref{cor1} we obtain
$$ad_{k-1}(\lambda(x,y,z))=ad_{k-1}(\delta_{\varphi}\rho(x,y,z)).$$
Therefore $\lambda(x,y,z)\in Z(\mathfrak{h})$. The hom-Lie color algebra $out(\mathfrak{h})= \frac{Der_{\alpha^{k}}(\mathfrak{h})}{Inn_{\alpha^{k}}(\mathfrak{h})}$ acts on $Z(\mathfrak{h})$, so $Z(\mathfrak{h})$ is a $\mathfrak{g}$-module by $\bar{\varphi}$ and $\delta_{\bar{\varphi}}$ is the cohomology differential.
Using \eqref{209} we have
$$\delta_{\bar{\varphi}}=\delta_{\varphi}\delta_{\varphi}\rho=[\rho,\rho]_{\wedge}=0,$$
therefore,
$$[\lambda]\in H^{3}(\mathfrak{g},Z(\mathfrak{h})).$$
We must show that the cohomology class $[\lambda]$ is independent of the choice of $\varphi$.
If we have $(\varphi,\rho)$ like  above and choose another linear lift $\varphi':\mathfrak{g}\rightarrow Der_{\alpha^{k}}(\mathfrak{h})$, then for a linear map $b:\mathfrak{g}\rightarrow\mathfrak{h}$ , $\varphi'(x)=\varphi(x)+ad_{k-1}(b(x))$. We set
$$\rho':\bigwedge^{2}_{graded}\mathfrak{g}\rightarrow\mathfrak{h},$$
$$\rho'(x,y)= \rho(x,y)+ (\delta_{\varphi}b)(x,y)+ [b(x),b(y)].$$
By calculations similar to Lemma \ref{lem1} we obtain
$$[\varphi'_{x},\varphi'_{y}]-\varphi'_{[x,y]}=ad_{k-1}(\rho'(x,y)),$$
and by the last part of Theorem \ref{th},
$$\lambda(\varphi,\rho)=\delta_{\varphi}\rho=\delta_{\varphi'}\rho'=\lambda(\varphi',\rho'),$$
so the cochain $\lambda$ remains unchanged. For a constant $\varphi$ let $\rho,\rho'$ be defined like
$$\rho,\rho':\bigwedge^{2}_{graded}\mathfrak{g}\rightarrow\mathfrak{h}, $$ $$[\varphi_{x},\varphi_{y}]-\varphi_{[x,y]}=ad_{k-1}(\rho(x,y))=ad_{k-1}(\rho'(x,y)).$$
Therefore
$$\rho-\rho':=\nu:\bigwedge^{2}_{graded}\mathfrak{g}\rightarrow Z(\mathfrak{h}).$$
It is obvious that
$$\lambda(\varphi,\rho)-\lambda(\varphi,\rho')=\delta_{\varphi}\rho=\delta_{\varphi}\rho'= \delta_{\bar{\varphi}}\nu.$$
 Now if there exists an extension inducing $\bar{\varphi}$, $\rho$ can be found like Theorem \ref{th} for each lift $\varphi$ such that $\lambda(\varphi,\rho)=0$. On the other hand, for a given $(\varphi,\rho)$ as described in Theorem \ref{th} such that
$$[\lambda(\varphi,\rho)]=0\in H^{3}(\mathfrak{g},Z(\mathfrak{h})),$$
there exists $\nu:\bigwedge^{2}_{graded}\mathfrak{g}\rightarrow Z(\mathfrak{h})$ such that $\delta_{\bar{\varphi}}\nu=\lambda$, therefore
$$ad_{k-1}((\rho-\nu)(x,y))=ad_{k-1}(\rho(x,y)),\delta_{\varphi}(\rho-\nu)=0.$$
Thus $(\varphi,\rho-\nu)$ satisfies the conditions of Theorem \ref{th}, so it describes an extension inducing $\bar{\varphi}$.
Now consider the linear lift $\varphi$ and a map $\rho:\bigwedge^{2}_{graded}\mathfrak{g}\rightarrow \mathfrak{h}$ satisfying \eqref{14} and \eqref{15} and note all $\rho'$s which satisfy this condition. We have
$$\rho-\rho':=\nu:\bigwedge^{2}_{graded}\mathfrak{g}\rightarrow Z(\mathfrak{h}),$$
and
$$\delta_{\bar{\varphi}}\nu=\delta_{\varphi}\rho-\delta_{\varphi}\rho'=0-0=0,$$
so $\nu$ is a 2-cocycle.\\
Moreover, analogous to Theorem \ref{th}, using a linear map $b:\mathfrak{g}\rightarrow\mathfrak{h}$ which preserves $\varphi$, i.e. $b:\mathfrak{g}\rightarrow Z(\mathfrak{h})$, one can use the corresponding data. Also $\rho'$ can be found using \eqref{18}
$$\rho'=\rho+\delta_{\varphi}b+\frac{1}{2}[b,b]_{\wedge}=\rho+\delta_{\bar{\varphi}}b.$$
Thus, it is just the cohomology class of $\nu$ that matters.
\end{proof}
The following corollary is an obvious consequence of the above theorem.
\begin{cor}
  Let $\mathfrak{g}$ and $\mathfrak{h}$ be two hom-Lie color algebras such that $\mathfrak{h}$ is Abelian. Then the cohomology classes of extensions of $\mathfrak{g}$ by $\mathfrak{h}$ are in one to one correspondence with all $(\varphi,[\rho])$s where $\varphi:\mathfrak{g}\rightarrow der(\mathfrak{h})$ is a hom-Lie color algebra homomorphism and $[\rho]\in H^{2}(\mathfrak{g},\mathfrak{h})$ is a graded cohomology class where $\mathfrak{h}$ is considered as a $\mathfrak{g}$-module by $\varphi$.
\end{cor}

\section{Acknowledgement}
The research in this paper has been supported by grant no. 92grd1m82582 of Shiraz University, Shiraz, IRAN.  A. Armakan is grateful to Mathematics and Applied Mathematics research environment MAM, Division of Applied Mathematics, School of Education, Culture and Communication at M{\"a}lardalen University, V{\"a}ster{\aa}s, Sweden for creating excellent research environment during his visit from September 2016 to March 2017.

% ----------------------------------------------------------------
%\bibliographystyle{amsplain}

\end{document}